\documentclass[a4paper,11pt,leqno]{amsart}

\usepackage[latin1]{inputenc}
\usepackage[T1]{fontenc}
\usepackage[english]{babel}

\usepackage{amsfonts}
\usepackage{amsmath}
\usepackage{amssymb}
\usepackage{amsthm}

\usepackage[usenames,dvipsnames,svgnames,table]{xcolor}

\newcommand{\R}{\mathbb{R}}

\newcommand{\N}{\mathbb{N}}

\newcommand{\inv}{^{-1}}
\newcommand{\diam}{{\operatorname{diam}}}
\newcommand{\eps}{\varepsilon}

\newtheorem{theorem}{Theorem}[section]{\bf}{\it}
\newtheorem{lemma}[theorem]{Lemma}{\bf}{\it}
{\bf}{\it}
\newtheorem{corollary}[theorem]{Corollary}{\bf}{\it}
{\bf}{\it}
{\bf}{\it}
\newtheorem*{theorem*}{Theorem}
\newtheorem{remark}{Remark}
\theoremstyle{remark}

\subjclass[2010]{30C65 (30L10, 57M12)}
\begin{document}

\title{Note on local-to-global properties of BLD-mappings}
\author{$\R$ami Luisto}
\address{Department of Mathematics and 
Statistics, P.O. Box 68 (Gustaf H\"allstr\"omin katu 2b), 
FI-00014 University of Helsinki, Finland}
\email{rami.luisto@helsinki.fi}
\thanks{The author is supported by the 
Academy of Finland project \#256228 and
the V\"ais\"al\"a foundation.}

\begin{abstract}
  We give a proof, based on Lipschitz quotient mappings,
  for the fact that limits
  of BLD-mappings between manifolds of bounded geometry
  are BLD. Furthermore we show that such mappings share 
  some properties of covering maps and especially
  have no asymptotic values.
\end{abstract}

\maketitle

\section{Introduction}

In this note we show that a uniform bound on the local
multiplicity for BLD-mappings gives rise to global properties
of BLD-mappings. BLD-mappings were first defined
by Martio and V\"ais\"al\"a
in \cite{MartioVaisala} as a special case of quasiregular mappings.
In this note we use the following definition which
was given as an equivalent condition in 
\cite[Theorem 2.16.]{MartioVaisala}.
A mapping $f \colon M \to N$ between metric spaces
is called an $L$-BLD-mapping, with $L \geq 1$, or a map of 
\emph{bounded length distortion},
if it is a continuous, open and discrete map for which
there exists a constant
$L \geq 1$ such that
for any path $\beta \colon [0,1] \to M$
we have
\begin{align}\label{eq:BLD}
  L\inv \ell(\beta)
  \leq \ell(f \circ \beta)
  \leq L \ell(\beta),
\end{align}
where $\ell( \cdot )$ is the length of a path.
Note that between length spaces
an $L$-BLD-mapping is always $L$-Lipschitz.

Another essential class of mappings in this
paper is the class of \emph{Lipschitz quotient mapping}.
A mapping $f \colon M \to N$
between two metric spaces 
is called an \emph{$L$-Lipschitz quotient}
($L$-LQ for short) mapping if 
\begin{align}\label{eq:LQ-def}
  B_N(f(x),L \inv r)
  \subset f B_M(x,r)
  \subset B_N(f(x),L r)
\end{align}
for all $x \in M$ and $r > 0$.
$L$-BLD-maps are $L$-LQ-mappings in length spaces,
see e.g.\
\cite[Remark 3.16(c)]{HeinonenRickman}.

One of our main tools is the following
characterization theorem of BLD-mappings.
For the definition of bounded geometry see Section 2.
\begin{theorem}\label{thm:Characterization}
  Let $M$ and $N$ be manifolds with bounded geometry.
  An $L$-BLD-mappings is a discrete $L$-LQ
  mapping, and a discrete $L$-LQ-mapping is always
  $L'$-BLD, where $L'$ depends only on the data.
  Furthermore, if the manifolds have strongly bounded geometry,
  $L' = L$.
\end{theorem}
The definitions and basic theory of BLD- and LQ-mappings
are not limited to the context of geodesic manifolds.
However, the methods we use in Theorem \ref{thm:Characterization}
rely on the locally Euclidean topology of manifolds and 
therefore cannot be directly translated to a more general setting.
We refer to \cite{HeinonenRickman} for a detailed discussion
on branched covers and BLD-mappings on generalized
manifolds.

It is
an open problem whether
$L$-Lipschitz quotient 
-mappings $\R^n \to \R^n$ are discrete (and hence BLD) mappings
for $n \geq 3$.
For positive result in dimension $n=2$ and for a
detailed discussion
see \cite{BJLPS}.

The first application of the characterization 
of Theorem \ref{thm:Characterization}
is a limit theorem for BLD-mappings.
It is a result of Martio and V\"ais\"al\"a
\cite[Theorem 4.7.]{MartioVaisala}
that locally 
uniform limits of $L$-BLD-mappings
between Euclidean domains are $L$-BLD. 
This result was generalized by Heinonen and Keith
in \cite[Lemma 6.2.]{HeinonenKeith}.
They showed
in a more general setting of complete generalized 
manifolds with some bounds on their geometry that
locally uniform limits of $L$-BLD-mappings
are $K$-BLD with some $K \geq L$ depending only on the
data. 
Martio and V\"ais\"al\"a proved the limit theorem 
by characterizing BLD-mappings as quasiregular
mappings with their Jacobian bounded from below and
Heinonen and Keith proved
the result by characterizing
BLD-mappings as locally regular mappings.
We prove the limit theorem by
using the previous characterization
theorem of BLD-mappings as LQ-mappings.

To further the analogy between these approaches
to the limit theorem,
we note that Heinonen and Keith show the limits
of regular maps to be quantitatively regular
and we show that the limits of $L$-LQ-mappings
are still $L$-LQ. In both arguments the loss of constants
happens when showing that the limit map, be it
locally regular or a discrete LQ, is a BLD-mappings.
To be more precise, we prove the following LQ limit theorem.
For the definition of $R$-uniformly $k$-to-one
mapping see the next section.
\begin{theorem}\label{theorem:LQlimit}
  Let $M$ and $N$ be $n$-manifolds with bounded geometry and
  let $(f_j)$ be a sequence of $R$-uniform $k$-to-one
  $L$-LQ-mappings $M \to N$ 
  converging locally uniformly to a continuous mapping
  $f \colon M \to N$.
  Then the limit map is a
  $R$-uniform $k$-to-one  
  $L$-LQ-mapping.
\end{theorem}

Combining the previous theorems we obtain as a 
corollary version
of a limit theorem of Martio and V\"ais\"al\"a for BLD-mappings
between manifolds with bounded geometry.
\begin{corollary}\label{coro:BLDLimit}
  Let $M$ and $N$ be $n$-manifolds with bounded geometry and
  let $(f_j)$ be a sequence of $L$-BLD-mappings $M \to N$
  converging locally uniformly to a continuous mapping
  $f \colon M \to N$.
  Then the limit map is an $L'$-BLD-mapping. 
  Furthermore, if the manifolds have strongly bounded geometry,
  $L' = L$.
\end{corollary}

Using the ideas of the proof of Theorem 
\ref{theorem:LQlimit}
the local bounds of the multiplicity can be used to show that
BLD-mappings share some properties of covering maps
which do not hold for
\emph{branched covers}, i.e.\ continuous, open and discrete 
mappings, or even quasiregular mappings. 
The following theorem shows that not only are
BLD-mappings branched covers but they
have a property similar to a covering property.
\begin{theorem}\label{thm:Martina}
  Let $M$ and $N$ be
  $n$-manifolds of bounded geometry
  and let $f \colon M \to N$ be
  an $L$-BLD-mapping. Then there 
  exists a radius 
  $R_{\text{inj}} > 0$ and constants $D > 0$ and
  $k \geq 1$
  depending only on $M$, $N$ and $L$
  such that for each $y \in N$
  and any $0 < r \leq R_\text{inj}$
  the pre-image $f \inv B_N(y,r)$ consists
  of pair-wise disjoint domains $U$ with $\diam (U) \leq D r $
  for which the restriction
  $f|_U \colon U \to B_N(y,r)$ is a surjective 
  $k$-to-one BLD-mapping.
\end{theorem}
This theorem especially implies that BLD-mappings
between manifolds of bounded geometry are
\emph{complete spreads} in the sense of \cite{Fox},
and branched coverings in the sense of 
\cite{Fox}
when the image of the branch set is closed;
see also \cite{Montesinos}.

Theorem \ref{thm:Martina} shows that
BLD-mappings between manifolds of bounded geometry
do not have asymptotic values.
A mapping $f \colon X \to Y$ between proper
metric spaces has an \emph{asymptotic value} 
at $y_0 \in Y$ if there exists a path 
$\beta \colon [0,1) \to X$ such that
$\lim_{t \to 1} f(\beta(t)) 
= y_0,
$
but the image $|\beta|$ of the path $\beta$ is not
contained in any compact set of $X$. 
\begin{corollary}\label{coro:BLDHasNoAsymptoticValues}
  Let $M$ and $N$ be
  $n$-manifolds of bounded geometry
  and let $f \colon M \to N$ be
  a BLD-mapping. Then $f$ 
  has no asymptotic values.
\end{corollary}
In this sense BLD-mappings differ greatly 
from quasiregular mappings.
For example Drasin constructs in 
\cite{Drasin97} an 
extremal example of a quasiregular mapping
$\R^3 \to \R^3$ for which every point in the range
of the map is an asymptotic
value.
For the definition and basic properties of quasiregular
mappings we refer to \cite{Rickman}.

Furthermore, from Corollary
\ref{coro:BLDHasNoAsymptoticValues} a simple
topological argument yields a Zorich
type theorem for BLD
mappings between manifolds of bounded geometry.
Note that the Zorich
type theorems for quasiregular mappings
hold only for conformally parabolic spaces,
see for example \cite{Zorich} and
\cite[Corollary 3.8]{Rickman};
for mappings of finite distortion see e.g.\
\cite{HolopainenPankka}
and 
\cite{KoskelaOnninenRajala}.
\begin{corollary}\label{coro:Zorich}
  Let $M$ and $N$ be $n$-manifolds of bounded geometry
  and let $ f \colon M \to N$ be a BLD-mapping.
  If $f$ is a local homeomorphism, then it is a covering map.
\end{corollary}

\bigskip
\noindent
\textbf{Acknowledgements.} The author would like to thank
his advisor Pekka Pankka for introducing him to the world
of BLD geometry. The author is also indebted to his advisor
and Martina Aaltonen for the many discussions concerning,
among a plethora of other things, the topics of this note.

We thank the referee for carefully reading the paper
and for helpful suggestions and remarks.

\section{Preliminary notions}

We say that a metric manifold $M$ which is
a complete length space has
\emph{bounded geometry} if 
there exists constants $C \geq 1$ and $R > 0$ 
such that for every point $x \in M$
there exists a $C$-bilipschitz mapping
$f \colon B_M(x,R) \to B_{\R^n}(0,R)$.
We say that $M$
has \emph{strongly bounded geometry}, if for any
$C \geq 1$
such a radius $R > 0$ can be found.
Here and in what follows we denote by $B_X(x,r)$
an (open) ball about $x \in X$ of radius $r$
in a metric space $(X,d)$.

Manifolds of bounded geometry 
are analogous to
complete Riemannian manifolds with
bounded Ricci curvature; indeed
complete Riemannian manifolds with bounded Ricci-curvature
have strongly bounded geometry, see
e.g.\ \cite[2.2]{Lelong}.
Note that since manifolds are locally Euclidean, they
are locally compact, and hence by Hopf-Rinow theorem
manifolds with bounded geometry are proper geodesic metric spaces
as locally compact and complete length spaces.

A mapping will be called
or \emph{$k$-to-one}, if any point in the
image has at most $k$ pre-images.
A mapping is said to be \emph{finite-to-one}
if it is $k$-to-one for some $k \in \N$.
For a discrete
map $f \colon X \to Y$ between metric spaces the quantity
\begin{align*}
  N(f(x),f,B_X(x,r)) := 
  \# \left( f \inv \{ f(x) \} \cap B_X(x,r) \right)\!,
\end{align*}
which we call the 
\emph{$r$-local multiplicity of $f$ around $x$},
is finite for all $x$ and $r$.
If this holds uniformly locally 
for a mapping $f \colon X \to Y$ between metric spaces,
i.e.\
\begin{align*}
  \# \left( f \inv \{ f(x) \} \cap B_X(x,R) \right) \leq k
\end{align*}
holds for all $x \in X$,
we say that the mapping is \emph{$R$-uniformly $k$-to-one}.

The set in which 
a branched cover $f$ fails to be a local homeomorphism
is called \emph{the branch set} and it is denoted 
by $B_f$. The branch set of a branched cover between
(generalized) manifolds is always small in a topological sense.
More precisely, let $ f \colon M \to N $ be a branched cover
between $n$-manifolds. Then $B_f$ has topological
dimension at most $n-2$; see \cite{Vaisala}.

The following lemma 
that implies that of BLD-mappings
are uniformly finite-to-one
is contained in \cite[Theorem 6.8]{HeinonenRickman}.
\begin{lemma}\label{lemma:GeometricIndexBound}
  Let $M$ and $N$ be 
  manifolds of bounded geometry and let 
  $f \colon M \to N$ be an $L$-BLD-mapping.
  Then $f$ is $R$-uniformly $k$-to-one
  with $R$ and $k$
  depending only on the data.
\end{lemma}
\begin{proof}
  By \cite[Theorem 6.8.]{HeinonenRickman},
  we have that
  \begin{align*}
    N(f(x), f , B_M(x,r))
    \leq (L c_M)^n
    \frac{\mathcal{H}^n ( B_M(x,\lambda r))}{\mathcal{H}^n(B_N(f(x),(\lambda -1)r/L c_N))}
  \end{align*}
  for $x \in M$, $r > 0$ and $\lambda > 1$,
  where $\mathcal{H}^n$ is the Hausdorff 
  $n$-measure, and the constants
  $c_M$ and $c_N$ are the quasi-convexity 
  constants of $M$ and $N$, respectively.
  We fix $\lambda = 2$ and note that 
  $c_M = c_N = 1$ in geodesic
  manifolds. Also, since $M$ and $N$ are
  manifolds of bounded geometry, 
  there exists a radius $R > 0$ and a constant $K \geq 1$
  such that when $r$ is small enough
  balls $B_M(x, 2 r )$ and $B_N(f(x),r/L)$ are $K$-bilipschitz
  equivalent to Euclidean balls 
  $B_{\R^n}(0,2r)$ and $B_{\R^n}(0,r/L)$, respectively.
  Thus, for small $r$, we have
  \begin{align*}
    N(f(x), f , B_M(x,r))
    \leq L^n K^{2n}\frac{\mathcal{H}^n ( B_{\R^n}(0, 2 r))}{\mathcal{H}^n(B_{\R^n}(0,r/L))}
    = 2^{n} L^{2n} K^{2n}.
  \end{align*}
\end{proof}

It is straightforward to see that
locally uniform limits of $L$-LQ-mappings 
are $L$-LQ; see e.g.\
\cite[Lemma 3.1.]{LeDonnePankka}.
We record this observation as a lemma.
\begin{lemma}\label{lemma:LQLimitIsLQ}
  Let $X$ and $Y$ be proper metric spaces
  and let $(f_j)$ be a sequence of
  $L$-LQ mappings converging locally
  uniformly to a continuous mapping
  $f \colon X \to Y$. Then $f$ is $L$-LQ.
\end{lemma}

We will also need the following result.
\begin{lemma}\label{lemma:LQBranchZeroMeasure}
  Let $G \subset \R^n$ be a domain and let
  $f \colon G \to \R^n$ be an $L$-LQ-mapping.
  Then $\mathcal{H}^n(B_f) = 0$.
\end{lemma}
\begin{proof}
  Outside the branch set $f$ is locally
  $L$-bilipschitz.
  Thus $|J_f(x)| \leq L^n$ for all points
  $x \notin B_f$ where $f$ is differentiable. 
  On the other hand,
  for any point $x \in B_f$, where
  $f$ is differentiable it holds
  by e.g.\ \cite[Lemma I.4.11]{Rickman}
  that $J_f(x) = 0$.

  Let $x \in G$ and 
  $0 < r < d(x,\partial G)$.
  By the coarea formula 
  \begin{align*}
    L^{-n} &\mathcal{H}^n(B_{\R^n}(f(x),r))
    \leq \mathcal{H}^n(fB_{\R^n}(x,r))
    \leq \int_{fB_{\R^n}(x,r)} N(y,f,B_{\R^n}(x,r))\,\operatorname{d}y \\
    &= \int_{B_{\R^n}(x,r)} |J_f| 
    = \int_{B_{\R^n}(x,r)\cap B_f} |J_f| + \int_{B_{\R^n}(x,r)\setminus B_f} |J_f| \\
    &\leq L^n\mathcal{H}^n(B_{\R^n}(x,r)\setminus B_f).
  \end{align*}
  Thus
  \begin{align*}
    \mathcal{H}^n(B_f \cap B_{\R^n}(x,r))
    \leq \mathcal{H}^n(B_{\R^n}(x,r))( 1 - L^{-2n}).
  \end{align*}
  We conclude that the set $B_f$ has no density points,
  hence $\mathcal{H}^n(B_f) = 0$.
\end{proof}

\section{Proofs of main theorems}

\subsection{A characterization of BLD-mappings}

We prove Theorem \ref{thm:Characterization}
by showing the two implications separately.
The fact that an $L$-BLD-mapping 
is $L$-LQ is noted already in 
\cite[Theorem 4.20]{MartioVaisala} and 
\cite[Remark 3.16(c)]{HeinonenRickman}.
The other direction follows as stated
from the following two lemmas.

\begin{lemma}\label{Lemma:Euclidean}
  Let $G \subset \R^n$ be a domain and let
  $f \colon G \to \R^n$ be a 
  discrete $L$-LQ 
  mapping, that is,
  the LQ-inequality \eqref{eq:LQ-def}
  holds for balls $B(x,r) \subset G$.
  Then $f$ is an $L$-BLD-mapping.
\end{lemma}
\begin{proof}
  A discrete LQ-mapping is a branched cover, so its branch set
  is closed and 
  has by \cite{Vaisala} topological dimension of at most $n-2$.
  Furthermore,
  $\mathcal{H}^n( B_f ) = 0$
  by Lemma \ref{lemma:LQBranchZeroMeasure}

  By the definition of LQ-mappings, it follows that
  $f$ is locally $L$-bilipschitz in $G \setminus B_f$.
  Especially for almost all $x \in G \setminus B_f$ we have
  \begin{align}\label{eq:LQdiff}
    L \inv 
    \leq \min_{| \mathbf{v} | \leq 1} | Df(x) \mathbf{v} |  
    \leq \| Df(x) \|
    \leq L
  \end{align}
  and $J_f(x) \neq 0$.
  Since $G \setminus B_f$ is locally connected,
  either $J_f > 0$ a.e.\ or $J_f < 0$ a.e.\ in $G \setminus B_f$.
  The branch set has measure zero so 
  almost everywhere in $G$ 
  the Jacobian is either strictly positive or strictly negative 
  and inequality \eqref{eq:LQdiff} holds. Now
  an argument with the mappings distributional
  derivative shows that
  $f$ satisfies the path-length inequality
  \eqref{eq:BLD},
  see for example the proof of
  \cite[Theorem 2.16]{MartioVaisala}.
\end{proof}

\begin{lemma}
  Let $f \colon M \to N$ be a discrete 
  $L$-LQ-mapping between manifolds of 
  $C$-bounded geometry. Then $f$ is $(C^4 L)$-BLD.
  If $M$ and $N$ have strongly bounded geometry,
  $f$ is $L$-BLD.
\end{lemma}
\begin{proof}
  For a discrete $L$-LQ-mapping $f \colon M \to N$
  between manifolds of boun\-ded geometry, we can
  use locally the $C$-bilipschitz charts given 
  by the bounded geometry condition. This means
  that locally our discrete $L$-LQ-mapping is conjugated
  to a discrete $(C^2 L)$-LQ-mapping between Euclidean
  domains.
  By Lemma \ref{Lemma:Euclidean}
  the map $f$ is locally $(C^4L)$-BLD.
  Since the BLD condition is local, the mapping
  $f$ is BLD.
  For manifolds of strongly bounded geometry
  we can take $C \to 1$ and the claim follows.
\end{proof}

Theorem \ref{thm:Characterization} is now proven.

\begin{remark}
  We note the following self improving property.
  Consider an
  LQ-mapping between manifolds of bounded
  geometry. The added requirement of discreteness 
  will make the LQ-map a BLD-mapping, so it will 
  especially be uniformly finite-to-one.
  It is not hard to see that the proof of
  Theorem \ref{thm:Martina} holds for
  a uniformly finite-to-one LQ-mapping, so
  the self-improvement
  continues as the mapping is then complete spread
  in the sense of Theorem \ref{thm:Martina}.
\end{remark}

\subsection{Limits of BLD-mappings}

The main reason we characterized BLD-mappings
as discrete LQ-mappings is because the limits
of $L$-LQ-mappings are easily seen to be 
$L$-LQ as we saw in Section 2. 
With the connection between BLD-mappings 
and discrete LQ-mappings 
we can thus show that locally uniform 
limits of $L$-BLD-mappings
are $L$-LQ.
To show the discreteness
of the limit map we need 
the mappings in the sequence
to be $R$-locally $k$-to-one with uniform constants
$R$ and $k$.
\begin{proof}[Proof of Theorem \ref{theorem:LQlimit}]
  We know by \ref{lemma:LQLimitIsLQ} that
  the limits of $L$-LQ-mappings are $L$-LQ.
  Suppose the limit map $f \colon M \to N$ is not 
  $R$-uniformly $k$-to-one.
  Then there
  exists a point $x_0 \in M$ such that 
  \begin{align*}
    \# \left( f \inv \{ f(x_0) \} \cap B_M(x_0,R) \right) 
    \geq k+1.
  \end{align*}
  Fix $k + 1$ points
  $z_0, \ldots, z_k$
  in $f \inv \{ y_0 \} \cap B_M(x_0, R)$,
  and denote the minimum of their pairwise distances
  by $\delta$. We fix
  \begin{align*}
    0 
    < \eps 
    < \min \left( \delta, 
    \min_{0 \leq j \leq k} \left( R - d(z_j , x_0) \right)
    \right)
  \end{align*}
  and let $j_0 \geq 1$ be such an index that
  for all $j \geq j_0$ 
  \begin{align*}
    \sup_{x \in B_M(x_0,R)} d(f_j(x),f(x))
    < \eps / (2L).
  \end{align*}
  Note that especially
  we have
  $f_j(z_k) \in B_N(y_0, \eps / (2L))$
  for all $j \geq j_0$. 
  The images
  of the balls $B_M(z_i, \eps / 2)$,
  $i = 0, \ldots, C$,
  under the $L$-LQ-mappings $f_j$,
  contain balls of radius 
  $\eps / (2L)$. 
  Thus
  \begin{align*}
    y_0 \in \bigcap_{i = 0}^{C} f_j B_M(z_i, \eps / 2).
  \end{align*} 
  In particular
  $
  N(f_j(x_0),f_j,B_M(x_0,R)) 
  \geq k+1.
  $
  for $j \geq j_0$.
  This contradicts
  Lemma \ref{lemma:GeometricIndexBound}.
  The limit mapping is thus $R$-uniformly $k$-to-one.
\end{proof}

\begin{proof}[Proof of Corollary \ref{coro:BLDLimit}]
  By Lemma \ref{lemma:GeometricIndexBound} and
  Theorem \ref{thm:Characterization} mappings
  $f_j \colon M \to N$
  are $R$-uniformly $k$-to-one $L$-LQ-mappings.
  Thus by Theorem \ref{theorem:LQlimit}
  the limit map $f \colon M \to N$ is a discrete $L$-LQ-mapping.
  The claim now follows from
  Theorem \ref{thm:Characterization}.
\end{proof}

\subsection{Global properties}
In this section we show how the local result of
Lemma \ref{lemma:GeometricIndexBound} implies
global properties of BLD-mappings. 
We recall first an elementary observation
on local surjectivity of
open maps between locally connected metric spaces.
Note that the lemma gives local surjectivity only
in precompact components of a pre-image of a given domain.
\begin{lemma}\label{lemma:LocalSurjectivity}
  Let $X$ and $Y$ be locally connected metric spaces,
  $f \colon X \to Y$ be a continuous 
  open map and $V \subset Y$ a domain.
  Then, for each precompact component
  $U$ of $f \inv V$,
  the restriction $f|_U \colon U \to V$
  is surjective.
\end{lemma}
\begin{proof}
  Let $U$ be a precompact component of $f \inv V$. 
  Since $X$ is locally connected and $f \inv V$
  is open, $U$ is an open set.
  Suppose $V \setminus fU \neq \emptyset$.
  Since $V$ is connected and intersects both $fU$ and its
  complement, there exists a point
  $z \in V \cap \partial fU$.
  Let $(z_n)$ be a sequence of points in $fU$ converging to $z$. For each
  $n \in \N$ we fix a point $y_n \in U \cap f \inv \{ z_n \}$.
  Since $\overline{U}$ is compact, there exists a subsequence
  $(y_{n_k})$ of $(y_n)$ converging to a point $y \in \overline{U}$.
  Since $f$ is continuous, $f(y) = z$.
  
  Since $z \in V$ and $y \in f \inv \{ z \}$,
  we may fix a component $U'$ of $f \inv V$ containing $y$.
  Since $U'$ is open and $y \in \overline{U}$, 
  we have $U \cap U' \neq \emptyset$. Thus $U' \subset U$
  and hence $y \in U$.
  On the other hand, 
  since $f$ is an open map, $fU$ is an open neighbourhood
  of $z = f(y)$ in $Y$,
  which implies
  $z \notin \partial fU$. This is a contradiction and the restriction 
  $f|_U \colon U \to V$ is surjective.
\end{proof}

Now we show that BLD-mappings between 
manifolds of bounded geometry are
similar to covering maps.
\begin{proof}[Proof of Theorem \ref{thm:Martina}.]
  Let $f$ be $R$-uniformly $k$-to-one
  and fix 
  \begin{align*}
    D
    := 2L(k+1),
    \quad
    \text{ and }
    \quad
    R_\text{inj} := \frac{R}{2L(k+1)}.
  \end{align*}
  Let $0 < r \leq R_{\text{inj}}$ and
  suppose there exists a point $y_0 \in N$
  such that the pre-image
  $f \inv B_N(y_0,r)$ contains
  a component $U$ of $f \inv B_N(y_0,r)$
  having diameter at least 
  $Dr$.
  Fix a point $x_0 \in U \cap f \inv \{ y_0 \}$
  and denote 
  $\eps = 2 L r$.
  Since $U$ is connected the 
  intersection $U \cap \partial B(x_0,s)$
  is non-empty for $ s \in (0, Dr / 2)$.
  Thus we can fix 
  $k + 1$ points 
  $z_0, \ldots, z_k \in U$
  having pair-wise distances
  of at least $\eps$ and for which
  $B_M(z_i, \frac{\eps}{2}) \subset B_M(x_0, Dr)$
  for $j = 0, \ldots, k$.

  As in the proof of
  Theorem
  \ref{theorem:LQlimit}
  we observe that
  \begin{align*}
    y_0
    \in \bigcap_{j = 0}^{k} f B_M(z_i, \frac{\eps}{2}),
  \end{align*}
  since $f$ is an $L$-LQ-mapping.
  Because 
  the balls $B_M(z_i, \frac{\eps}{2}) \subset B_M(x_0, Dr)$,
  $j = 0, \ldots, k$, are pairwise
  disjoint and
  $Dr \leq D R_\text{inj} = R$, we have
  \begin{align*}
    N(f(x_0),f,B_M(x_0,R)) 
    \geq N(f(x_0),f,B_M(x_0,Dr)) 
    \geq k+1.
  \end{align*}
  This contradicts the choice of $R$ and $k$.
  Thus $\diam ( U ) < Dr$.

  By Theorem \ref{thm:Characterization} the mapping $f$
  is an $L$-LQ-mapping and as such a continuous open map.
  It follows from the first inclusion of the equation \eqref{eq:LQ-def}
  defining LQ-mappings that an LQ-mapping is always surjective;
  for any fixed $x \in M$ the images of 
  the balls $B(x,r)$, $r > 0$, cover all of $N$.
  Especially the pre-image $f \inv B_N(y,r)$ will be non-empty
  for all $y \in N$ and $r < R_\text{inj}$,
  and by the first part of the proof its each component $U$ is bounded.
  Since $M$ is a proper metric space as a manifold of bounded geometry,
  the set $U$ is a precompact set as a bounded subset of 
  a proper metric space. Also note that $M$ and $N$ are locally connected 
  since they are manifolds.
  Thus the surjectivity of the restriction 
  $f|_U \colon U \to B_N(y, r)$
  follows from Lemma
  \ref{lemma:LocalSurjectivity}.
\end{proof}
\begin{remark}
  Note that for mappings $f \colon M \to N$
  satisfying the conclusion of Theorem \ref{thm:Martina},
  it especially holds that for small enough $r$ and for any $x \in M$
  the $x$-component $U(x,f,r)$ of $B_N(f(x),r)$ has diameter at most
  $D r$. In a doubling metric space, such as manifolds of bounded
  geometry, this implies that the mapping is locally regular.
  Thus
  to prove a qualitative version of the limit theorem, we
  could first show that the limit of a sequence of $L$-BLD-mappings is 
  an $R$-uniformly $k$-to-one $L$-LQ-mapping, then use Theorem \ref{thm:Martina}
  to show that such a map is regular, and then use methods
  of Heinonen, Keith and Rickman to show that a regular map is BLD.
\end{remark}

From the the previous result it follows that a BLD-mapping has no
asymptotic values.
\begin{proof}[Proof of Corollary \ref{coro:BLDHasNoAsymptoticValues}.]
  By Theorem \ref{thm:Martina} the lifts of sufficiently
  short paths are bounded. This forbids asymptotic values.
\end{proof}

From the lack of asymptotic values it follows that
a locally homeomorphic BLD-mapping between manifolds
of bounded geometry is a covering map.
\begin{proof}[Proof of Corollary \ref{coro:Zorich}]
  Since universal covers $\tilde M$ and 
  $\tilde N$ 
  are also manifolds of bounded geometry, a lift
  $\tilde f \colon \tilde M \to \tilde N$ of $f$
  has no asymptotic values by 
  Corollary \ref{coro:BLDHasNoAsymptoticValues}.
  A standard path-lifting argument of local homeomorphism
  without asymptotic values shows that $\tilde f$
  is invertible and thus a homeomorphism. Thus
  $f \colon M \to N$ is a covering map.
\end{proof}

\begin{remark}
As a final remark we note that results 
in this last section generalize to any proper path-metric
space where the conclusion of Lemma \ref{lemma:GeometricIndexBound}
holds. Alternatively, these result will 
hold true for the class of 
$R$-uniformly $k$-to-one
$L$-LQ-mappings.
\end{remark}



\def\cprime{$'$}\def\cprime{$'$}

\end{document}